\documentclass{article}

\usepackage{amsmath, amssymb, amsthm,amsfonts}
\usepackage{mathrsfs}
\usepackage{color}
\usepackage{cite}
\usepackage[colorlinks,linkcolor=red,anchorcolor=blue,citecolor=blue]{hyperref}

\usepackage{enumerate}

\newtheorem{theorem}{Theorem}[section]

\newtheorem{lemma}{Lemma}[section]
\newtheorem{proposition}{Proposition}[section]
\newtheorem{corollary}{Corollary}[section]
\newtheorem{remark}{Remark}[section]

\author{Weiyuan Qiu, Lingrui Wang}
\date{November 25, 2025}

\title{Bounded Fatou components of cosine functions}

\begin{document}

\maketitle

\begin{abstract}
    We constructed Yoccoz puzzle for cosine functions $f(z)=ae^z+be^{-z}$ with bounded post-critical set, and proved that a Fatou component is a Jordan domains if it is bounded and is not eventually a Siegal disk. We proved that $f$ is renormalizable if a critical value escapes to $\infty$. Finally, we obtained the local connectivity of $J(f)$.

\end{abstract}

\section{Introduction}\label{sec: Introduction}

In the dynamics of entire function $f$, the complex plane can be divided into two sets: the Fatou set $F(f)$ and the Julia set $J(f)$, according to whether the iteration sequence forms a normal family locally. By the classification of Fatou components, the dynamics on $F(f)$ is clear. However, even for the quadratic polynomials, the dynamics on the Julia sets is quite complicated. By Carath\'eodory's theorem, if the boundary of a Fatou component $U$ is locally connected, then the dynamics of $f$ on $\partial U$ can be conjugate to a mapping of much simpler form. Therefore, the local connectivity of the boundaries of Fatou components and the Julia set is an important topic in complex dynamics. 

The local connectivity of Julia sets has been studied for decades: Hyperbolic, subhyperbolic, semi-hyperbolic, geometrically finite and Collet-Eckmann polynomials always have locally connected Julia sets (\cite{DH84, CJY94, TY96, GS98}). These results have the analogy for transcendental entire functions (\cite{BM02, BFR15, Par22, ARS22}). Essentially, all these results assume that the mapping is expanding on a neighborhood of its Julia set. So a natural question is whether the Julia set (or the boundaries of Fatou components) is locally connected without the assumption on the expandness of $f$ on its Julia set. 

In early 1990s, Yoccoz (\cite{Hub93}) introduced the puzzle technique for quadratic polynomials. With this technique, he was able to study the local connectivity of quadratic polynomials without assuming the expandness on the Julia sets. Later, his results were extended to all polynomials by Roesch and Yin (\cite{RY08}). 

The puzzle technique is itself a powerful tool in the dynamics of polynomials. For example, it is involved in the study of the density of hyperbolicity in real polynomials (\cite{KSS07}) and the local connectivity of Julia sets of Newton mappings (\cite{Roe08}). The construction of puzzles in the polynomial cases is based on the fact that $\infty$ is a superattracting fixed point and the immediate basin of $\infty$ enjoys the structure of external rays and equipotentials. However, because of the absence of equipotentials, it becomes extremely difficult to construct Yoccoz puzzles for transcendental entire functions. 

In this paper, we considered a simple family of transcendental entire functions, the cosine family $\{f(z)=ae^z+be^{-z}: a,b \in \mathbb{C} \setminus \{0\}\}$. With the puzzles, we obtained the local connectivity of the boundaries of Fatou components:
\begin{theorem}\label{thm: bounded Fatou components are Jordan domains}
    Let $f(z)=ae^z+be^{-z}$ be a cosine function. Suppose that $P(f)$ is bounded, $U$ is a bounded Fatou component, which is not eventually mapped to a Siegel disk. Then $U$ is a Jordan domain. 
\end{theorem}

It is worth noting that there are also some results about the local connectivity of the boundaries of Siegel disks. For example, for quadratic polynomials, if the rotation number of the Siegel disk is of bounded type, then the boundary of the disk is locally connected (\cite{Pet96}). And the local connectivity of Siegel disks holds for some families of transcendental entire functions (\cite{Zhang05,Yang13,Zhang16,ZFS20}).

The cosine function $f(z)=ae^{z}+be^{-z}$ has two critical values and no asymptotic values. So it shares some similarities with cubic polynomials. It is known that for a polynomial $p$ with at least one escaping critical orbit, or equivalently with $J(p)$ disconnected, $J(p)$ is a Cantor set if and only if the critical components of $K(p)$ is not periodic (\cite{BH92,QY09}). Here, $K(p)$, which is called the filled-in Julia set, is the complement of the immediate basin of $\infty$. In other words, if $p$ has periodic Fatou components and $J(p)$ is disconnected, then $p$ is renormalizable. We have the analogy for cosine functions:
\begin{theorem}\label{thm: cosine functions with an escaping critical value is renormalizable}
    Let $f$ be a cosine function. Suppose that the Fatou set is non-empty, and one of the critical values escapes to $\infty$. Then $f$ is renormalizable.
\end{theorem}

In this case, the local connectivity of Fatou components can be deduced easily.
\begin{corollary}\label{cor: Fatou components are locally connected in renormalizable case}
    Let $f$ be a cosine function. Suppose that the Fatou set is non-empty, and one of the critical values escapes to $\infty$. If $f$ has no Siegel disks, then all Fatou components of $f$ are Jordan domains.
\end{corollary}

Finally, with Theorem \ref{thm: bounded Fatou components are Jordan domains} and Corollary \ref{cor: Fatou components are locally connected in renormalizable case}, we conclude the local connectivity of $J(f)$:
\begin{theorem}\label{thm: local connectivity of Julia set}
    Let $f$ be a cosine function. Suppose that $F(f)$ is non-empty and has no Siegel disks. Suppose that a critical value $v$ lies in $F(f)$. Suppose that $f$ satisfies one of the following conditions:
    \begin{enumerate}
        \item the critical value $-v \in F(f)$ and is not in the same Fatou component as $v$,
        \item the critical value $-v$ escapes to $\infty$, or
        \item $-v$ has bounded orbit and there exists an integer $N \ge 1$, such that for every periodic Fatou component $U$, the closure of every strictly preperiodic component of $f^{-N}(U)$ is disjoint from the $\omega$-limit set of $-v$.
    \end{enumerate}
    Then $J(f) \cup \{\infty\}$ is locally connected. 
\end{theorem}

This paper is structured as following:
Section \ref{sec: basic settings} stated some basic settings and proved a dichotomy of the boundedness of Fatou components; Section \ref{sec: construction of puzzles} constructed the Yoccoz puzzle for cosine functions with bounded post-critical set; Section \ref{sec: local connectivity of bounded Fatou components} proved Theorem \ref{thm: bounded Fatou components are Jordan domains} using the puzzle; Section \ref{sec: renormalization of cosine functions} proved Theorem \ref{thm: cosine functions with an escaping critical value is renormalizable} and Corollary \ref{cor: Fatou components are locally connected in renormalizable case}; and Section \ref{sec: local connectivity of Julia sets} proved Theorem \ref{thm: local connectivity of Julia set}.

\section{Basic settings}\label{sec: basic settings}
This section gives some basic assumptions on the cosine function $f(z)=ae^{z}+be^{-z}$. $f$ can be written into the following form:
\begin{equation*}
    f(z)=\frac{v}{2}(e^{z-u}+e^{-(z-u)}),
\end{equation*}
with critical points $u+k\pi i$, $k \in \mathbb{Z}$ and critical values $\pm v$, such that $f(u+2k\pi i)=v$, $f(u+(2k+1)\pi i)=-v$, $k \in \mathbb{Z}$. Thus $f$ is a function with only two critical values and no asymptotic values. By the results of Eremenko-Lyubich (\cite{EL92}, Theorem 1 \& Theorem 3), $F(f)$ has no wandering domains and the only possibilities of periodic Fatou components are: basins of attracting periodic points, basins of parabolic periodic points, or Siegel disks. 

The positions of critical points and critical values satisfy the following dichotomy:
\begin{lemma}\label{lm: dichotomy of the positions of critical points}
    Let $f$ be a cosine function. Suppose that $F(f)$ is non-empty, $f$ has no Siegel disks, and the critical orbits either escape to $\infty$ or remain bounded. Then one of the following holds.
    \begin{enumerate}
        \item The critical points of $f$ lie in the unique Fatou component, and all Fatou components are unbounded;
        \item Every component of $F(f)$ contains at most one critical point of $f$, and every component of $F(f)$ is bounded. 
    \end{enumerate}
\end{lemma}
\begin{proof}
    We will discuss the positions of critical values of $f$. Let $B$ be a Fatou component of $f$. Since $F(f)$ has no wandering components, we assume that $B$ is periodic. Moreover, we consider an invariant Fatou component $B$, i.e. $f(B)=B$. 

    \textbf{Case 1.} $B$ contains both critical values. Take a curve $\gamma$ in $B$ connecting the two critical values. Then $f^{-1}(\gamma)$ is a curve in $B$, connecting all critical points of $f$. Thus $B$ is unbounded. Since every cycle of Fatou components is associated with at least one critical orbit, $f$ has no other cycles of Fatou components. This implies that every Fatou component $B'$ is eventually mapped to an unbounded Fatou component $B$. $B'$ is unbounded since $f$ is an entire function.

    \textbf{Case 2.} $B$ contains only one critical value, say $v$. In this case, $B$ contains only one critical point, say $u$, and $u+2k\pi i \in B+2k\pi i$, $k \in \mathbb{Z}$, which are distinct Fatou components. 

    Suppose by contradiction that $B$ is unbounded. Let $\delta:(0,+\infty) \to \mathbb{C}$ be a curve in $B$ with $\delta(t) \to \infty$ as $t \to +\infty$. Since the orbits of critical values $\pm v$ either are bounded or tend to $\infty$, we can always find a dynamic ray (see \ref{sec: construction of puzzles} for the definition) $\gamma$, which converges at some point on $\partial B$. Then $\pm (\gamma \pm 2\pi i)$ converges to a point on $\partial B + 2\pi i$ or $\partial B-2\pi i$. 

    \textbf{Claim 1. }$|\mathrm{Re} \; \delta| \to +\infty$. Otherwise, $\delta$ is contained in a vertical strip $\{z \in \mathbb{C}: |\mathrm{Re} \; z| \le M\}$ for some $M>0$. We may assume that $\delta \subset \{z \in \mathbb{C}: \mathrm{Re}(z-u)<0\}$, and $\mathrm{Im} \; \delta(t) \to +\infty$, as $t \to +\infty$. Thus at least one of $(\gamma \pm 2\pi i)$ is contained in the strip $\{z \in \mathbb{C}: |\mathrm{Re} \; z| \le M\}$. This contradicts Lemma \ref{lm: properties of dynamic rays}. 

    \textbf{Claim 2. }$\mathrm{Im} \; \delta(t)$ is bounded. Again we can choose a dynamic ray $\gamma$ in $\{z \in \mathbb{C}: \mathrm{Re}(z-u)<0\}$, which converges to some point on $\partial B$. Then $B \cap \{z \in \mathbb{C}: \mathrm{Re}(z-u)<0\}$ is contained in the domain bounded by $\gamma+2\pi i$, $\gamma-2\pi i$, and the vertical line $\{z \in \mathbb{C}: \mathrm{Re}(z-u)=0\}$. By Lemma \ref{lm: properties of dynamic rays}, $\gamma$ is eventually horizontal. Thus $B \cap \{z \in \mathbb{C}: \mathrm{Re}(z-u)<0\} \subset \{x+iy: x<0, (k_0-1)2\pi \le y \le (k_0+1)2\pi\}$ for some $k_0$. By the symmetry of $B$, there exists $M_0>0$, such that $B \subset S=:\{z \in \mathbb{C}: -M_0 < \mathrm{Im} \; z< M_0\}$. In particular, $\mathrm{Im} \; \delta(t)$ is bounded. 

    Finally, we prove that $B$ is bounded. Since $B$ is $f$-invariant, $B \subset S \cap f^{-1}(S)$. For every $\epsilon>0$, there exists $M>0$, such that for $z \in S \cap f^{-1}(S)$ with $|\mathrm{Re} \; z|>M$, we have $|\mathrm{Im}(z-u) + \mathrm{arg} \; v|<\epsilon$. Choose $M' >M$, such that $e^{M'}>6M'$. For every $z \in S \cap f^{-1}(S)$, if $|\mathrm{Re} \; z| \ge M'$, since $|\mathrm{Im}(z-u)+\mathrm{arg} \; v|<\epsilon$, we have 
    \begin{equation*}
        |f(z)| \ge \frac{e^{|\mathrm{Re} \; z|}}{2}.
    \end{equation*}
    Take $M'$ sufficiently large, such that $\sqrt{e^{2M'}/4-M^2_0} > e^{M'}/3$. For $z \in S \cap f^{-1}(S)$, we have $|\mathrm{Re} \; f(z)| \ge e^{|\mathrm{Re} \; z|}/2 \ge e^{2M'}/4$. Thus 
    $$|\mathrm{Re} \; f(z)| \ge \sqrt{e^{2M'}/4-M^2_0} \ge e^{M'}/3 >M'.$$
    In particular, for $z \in B$, $|\mathrm{Re} \; z| > M'$, we have $|\mathrm{Re} \; f(z)| \ge M'$, and $f(z) \in B \subset S \cap f^{-1}(S)$. Inductively, $|\mathrm{Re} \; f^n(z)| >M'$, contradicting that $\{f^n(z)\}$ tends to an attracting or parabolic fixed point. Therefore, $B$ is bounded.
\end{proof}

If the critical values lie in the same Fatou component, then the Fatou components of $f$ are unbounded by Lemma \ref{lm: dichotomy of the positions of critical points}. By a result of Baker and Weinreich (\cite{BW91}), the boundary of any Fatou component is not locally connected at any point:
\begin{lemma}[\cite{BW91}]\label{lm: not locally connected anywhere}
    Let $f$ be a transcendental entire function and $U$ be an unbounded Fatou component of $f$. If $f^n$ does not tend to $\infty$ on $U$, then $\partial U$ is not locally connected at any point.
\end{lemma} 

Therefore, we may assume that the periodic Fatou component $B$, which is an attracting or parabolic basin, contains a critical point $u$, so that $f(B)$ contains the critical value $v$. Moreover, if $B$ is the basin of an attracting periodic point, we assume that it is a superattracting basin by applying a quasiconformal surgery.

\section{Construction of puzzle}\label{sec: construction of puzzles}
In this section, we will construct Yoccoz puzzles for the cosine function $f$. Based on the assumptions in \S \ref{sec: basic settings}, $B$ is a bounded periodic Fatou component of period $m \ge 1$, containing a critical point $u$. It is the basin of either a superattracting or parabolic periodic point. 

\subsection{Construction of unbounded puzzle}\label{ss: construct unbounded puzzle pieces}

Let us recall the construction of puzzles for polynomials: a collection of periodic external and internal rays, which have the common landing points, an equipotential outside the filled-in Julia set, and equipotentials inside the periodic Fatou components, which separate the boundaries of the Fatou components and the critical orbits in the cycle of Fatou components. 

In order to construct the puzzles for $f$, we also need some periodic rays landing on $\partial B$, and equipotentials separating $B$ from $\infty$ and the critical orbit. 

\textbf{Internal rays:} The structure of internal rays and equipotentials in bounded Fatou components exists for the basin of both superattracting and parabolic periodic points. If $B$ is a basin of superattracting periodic point, then it is well-known that there exists a conformal mapping $\phi_B: B \to \mathbb{D}$ (since $B$ contains only one critical point), such that $\phi_B \circ f^m \circ \phi^{-1}_B(z)=z^2$ for $z \in \mathbb{D}$. Define the internal ray of angle $\theta \in [0,1)$ by $\phi^{-1}_B((0,1)e^{2\pi i \theta})$, and the equipotential of potential $t$ by $\phi^{-1}_B(\{t e^{2\pi i \alpha}: \alpha \in [0,1)\})$. If $B$ is a basin of parabolic periodic point, the existence of internal rays and equipotentials can be found in \cite{Roe10}. The internal ray of angle $\theta$ is denoted by $R_B(\theta)$, and the equipotential of potential $t$ is denoted by $E_B(t)$. 

\textbf{Dynamic rays:} For the cosine function $f$, $\infty$ is no longer a superattracting fixed point, rather than an essential singularity. So there is no external ray. But a subset of $J(f)$, the escaping set $I(f)$, can still be described by rays (see \cite{RS08}). 

Let $\Gamma$ be the union of the segment connecting $\pm v$ and the vertical ray connecting $v$ to $\infty$,  and let $P_{j,k}$, $j \in \{0,1\}$, $k \in \mathbb{Z}$ be the components of $f^{-1}(\mathbb{C} \setminus \Gamma)$. Here the half strip $P_{0,k}$ lies in the right half plane $\{z \in \mathbb{C}: \mathrm{Re} \; (z-u)>0\}$ while $P_{1,k}$ lies in the left half plane $\{z \in \mathbb{C}: \mathrm{Re} \; (z-u)<0\}$. And the segment $[u_{2k}, u_{2k+2}]$ lies on the boundary of $P_{j,k}$, $j=0,1$, $k \in \mathbb{Z}$. 

$f: P_{j,k} \to \mathbb{C} \setminus \Gamma$ is conformal. Then every point $z \in I(f)$, for sufficiently large $n$, $f^n(z)$ lies in a unique $P_{j_n, k_n}$. So $z$ is associated with a unique sequence $\underline{s} \in (\{0,1\} \times \mathbb{Z})^{\mathbb{N}} =:\Sigma$, which is called the address of $z$. Conversely, every sequence $\underline{s} \in \Sigma$ coresponds to a subset $g_{\underline{s}}$ of $I(f)$. Here $g_{\underline{s}}$ is either empty or a $C^{\infty}$-curve tending to $\infty$. Then $g_{\underline{s}}$ is called a dynamic ray, and $\underline{s}$ is called its address. Intuitively, for $z \in g_{\underline{s}}$, $n$ sufficiently large, $f^{n}(z) \in P_{s_n}$, where $s_n=(j_n, k_n)$ is the n-th entry of $\underline{s}$. 

The dynamic ray $g_{\underline{s}}$ can be written as a map $g_{\underline{s}}: (t_{\underline{s}},+\infty) \to I(f)$. Here $t_{\underline{s}} \ge 0$ is minimal such that $g_{\underline{s}}(t_{\underline{s}},+\infty)$ can be defined. The dynamic rays have the following properties:
\begin{lemma}[\cite{RS08}, Theorem 4.1 \& Proposition 4.3]\label{lm: properties of dynamic rays}
    Let $g_{\underline{s}}:(t_{\underline{s}},+\infty) \to I(f)$ be a (non-empty) dynamic ray, $F(t)=e^t-1$, and $\sigma: \Sigma \to \Sigma$ be the shift on $\Sigma$. Then 
    \begin{enumerate}
        \item $f(g_{\underline{s}}(t))=g_{\sigma(\underline{s})}(F(t))$;
        \item Suppose that $\underline{s}=(s_0, s_1, \ldots)$, $s_n=(j_n, k_n)$, $n \ge 0$. Then 
        \begin{equation*}
            g_{\underline{s}}(t)=
        \begin{cases}
            t-\log a+2k_0\pi i +O_{\underline{s}}(e^{-t}), & \quad t \to +\infty, \quad \text{if} \quad j_0=0 \\
            -t+\log b +2k_0\pi i +O_{\underline{s}}(e^{-t}), & \quad t \to +\infty, \quad \text{if} \quad j_0=1
        \end{cases}.
        \end{equation*}
    \end{enumerate}
\end{lemma}

\begin{remark}\label{rmk: eventually contained}
    Let $U \subset \mathbb{C}$ be an unbounded domain and $\gamma:(0,+\infty) \to \mathbb{C}$ be a curve tending to $\infty$. We say $\gamma$ is \textbf{eventually contained} in $U$, if for sufficiently large $T>0$, $\gamma((T, +\infty)) \subset U$. Therefore, Lemma \ref{lm: properties of dynamic rays}, 2 implies that a dynamic ray $g_{\underline{s}}$ is eventually contained in $P_{j_0,k_0}$ provided that $\underline{s}=((j_n,k_n))_{n \ge 0}$. 
\end{remark}

The dynamic ray $g_{\underline{s}}$ is called to land at some point $z_0 \in \mathbb{C}$, if the following limit holds:
\begin{equation*}
    \lim_{t \to t_{\underline{s}}} g_{\underline{s}}(t)=z_0.
\end{equation*}

For polynomials, there is a well-known landing theorem: if the Julia set is connected, then the external rays with rational angles always land, and every repelling or parabolic periodic point is the landing point of at least one and at most finitely many external rays. When $f$ is a cosine function with the post-critical set bounded, a similar landing theorem holds:
\begin{lemma}[\cite{BR20}, Theorem 1.4]\label{lm: landing theorem}
    Let $f$ be a cosine function whose post-critical set $P(f)$ is bounded. Then every periodic dynamic ray $g_{\underline{s}}$ lands at a repelling or parabolic periodic point. And conversely, every repelling or parabolic point is the landing point of at least one and at most finitely many periodic dynamic rays. 
\end{lemma}

\begin{remark}
    Theorem 1.4 in \cite{BR20} is much more general. It concerns all `criniferious' transcendental entire functions. But we just need the form above. 
\end{remark}

\textbf{Yoccoz Puzzles:} Now we are ready to construct the puzzles. Similar to polynomials, we start with a periodic internal ray $R_B(\theta)$ of period $p$ in $B$, such that it lands at a repelling periodic point $z_0 \in \partial B$, and it avoids the orbit of $-2v$. By Lemma \ref{lm: landing theorem}, there is a periodic dynamic ray $g_{\underline{s}}$ landing at $z_0$, which is also of period $p$. Let $\Gamma_0=\bigcup_{j=0}^{p-1} f^j(\overline{R_B(\theta) \cup g_{\underline{s}}})$, which is an $f$-forward invariant graph. Let $\Gamma_{\infty}=\bigcup_{n \ge 0} f^{-n}(\Gamma_0)$. Then $\Gamma_{\infty}$ is completely invariant. Choose an equipotential $E_B(t) \subset B$, separating $\partial B$ from the critical point $u$. 

Let $E_0=\bigcup_{j=0}^{m-1} f^j(E_B(t))$. The components of $\mathbb{C} \setminus (E_0 \cup \Gamma_0)$ intersecting with $J(f)$ are called the puzzle pieces of depth $0$. Inductively, if $Q_n$ is a puzzle piece of depth $n$, then every component of $f^{-1}(Q_n)$ is called a puzzle piece of depth $n+1$. For $z \in J(f) \setminus \Gamma_{\infty}$, the puzzle piece of depth $n$ containing $z$ is denoted by $Q_n(z)$. Then it is easy to see that $Q_{n+1}(z) \subset Q_n(z)$, and if $Q_n$ and $Q'_n$ are both puzzle pieces of depth $n$, then they are either the same or disjoint.  

\subsection{Modification of unbounded puzzle pieces}\label{ss: modification of unbounded puzzles}
The puzzle pieces constructed in \S \ref{ss: construct unbounded puzzle pieces} are all unbounded. However, since we only concern about a bounded Fatou component $B$, we can cut down the puzzle pieces into bounded ones. 

\begin{lemma}\label{lm: separated along the imaginary axis}
    Let $g_{\underline{s}} \subset \Gamma$ be a dynamic ray landing at $\alpha \in U$, where $U$ is a periodic Fatou component containing a critical value. Then there exist two dynamic rays $g_{\underline{s'}}$, $g_{\underline{s''}} \subset f^{-1}(g_{\underline{s}})$, satisfying:
    \begin{enumerate}
        \item $g_{\underline{s'}}$, $g_{\underline{s''}}$ land at the boundary of the same component of $f^{-1}(U)$;
        \item $\underline{s'}=((0,k'_0), \underline{s})$, $\underline{s''}=((1,k''_0), \underline{s})$.
    \end{enumerate}
    In other words, the tails of these two rays enter both half planes $\{z \in \mathbb{C}: \mathrm{Re} (z-u)>0\}$ and $\{z \in \mathbb{C}: \mathrm{Re} (z-u)<0\}$.  
\end{lemma}
\begin{proof}
    The preimages of $g_{\underline{s}}$ are rays with addresses $((j,k), \underline{s})$, $j=0,1$, $k \in \mathbb{Z}$. Denote by $\alpha_{j,k}$ the landing point of $g_{((j,k), \underline{s})}$, $j=0,1$, $k \in \mathbb{Z}$. Note that $\alpha_{j,k+1}=\alpha_{j,k}+2\pi i$. Thus $g_{((0,k), \underline{s})}$ cannot land at the boundary of the same Fatou component. But for every component $U'$ of $f^{-1}(U)$, $\partial U$ contains exact two preimages of $\alpha$ since $U$ has a unique critical point. Thus $\partial U$ contains $\alpha_{0,k}$ and $\alpha_{1,k'}$ for some $k, k' \in \mathbb{Z}$. 

    Therefore, $\underline{s'}=((0,k), \underline{s})$, $\underline{s''}=((1,k'), \underline{s})$ are the desired addresses. 
\end{proof}

\textbf{Bounded puzzle pieces:} Consider the graph $\Gamma_1=f^{-1}(\Gamma_0)$. By Lemma \ref{lm: properties of dynamic rays}, the dynamic rays in $\Gamma_1$ are eventually horizontal. By Lemma \ref{lm: separated along the imaginary axis}, the dynamic rays in $\Gamma_1$ are eventually contained in both half planes $\{z \in \mathbb{C}: \mathrm{Re} (z-u)>0\}$ and $\{z \in \mathbb{C}: \mathrm{Re} (z-u)<0\}$. Thus the dynamic rays in $\Gamma_1$ together with the internal rays landing at the same point separate $\mathbb{C}$ into countably many strips, who are eventually horizontal and have uniformly bounded height. 

Since $B$ and $P(f)$ are bounded, we can cover them by finitely many strips above. Therefore, there is a horizontal strip $S=\{z \in \mathbb{C}: |\mathrm{Im} \; z| \le M\}$, containing $\bigcup_{j=0}^{m-1} f^j(B)$ and $P(f)$. Moreover, we may choose $S$ such that the puzzle pieces of depth $1$, which intersect with $\partial f^j(B)$, $j=0, \ldots, m-1$ or contain points in $P(f)$, are contained in $S$. Thus every puzzle piece of depth $n \ge 1$, intersecting with $\partial f^j(B)$, $j=0, \ldots, m-1$ or containing points in $P(f)$, are contained in $S$. 

For $M'>0$, $R=\{z \in S: |\mathrm{Re} \; (z-u)| \le M'\}$. $E=f(R)$ is a domain surrounded by an ellipse with foci $\pm v$. Choose $M'$ sufficiently large, then $R \Subset E$. For a puzzle piece $Q_1$ of depth $1$ intersecting $\partial f^j(B)$, let $P_1$ be the interior of $Q_1 \cap E$. $P_1$ is a modified puzzle piece of depth $1$. Then every component $P_2$ of $P_1$, which intersects with $\partial f^j(B)$ or contains points in $P(f)$, lies in $R$. This implies that $P_2$ is contained in some modified puzzle piece of depth $1$. Inductively, if $P_n$ is a modified puzzle piece of depth $n$, then every component of $f^{-1}(P_n)$, which intersects with $\partial f^j(B)$, $j=0, \ldots, m-1$ or contains points in $P(f)$, is called a modified puzzle piece of depth $n+1$. For $z \in \partial f^j(B)$, $j=0, \ldots, m-1$, or $z \in P(f)$, which is not in $\Gamma_{\infty}$, denote by $P_{n}(z)$ the modified puzzle piece of depth $n$ containing $z$. Then $P_{n+1}(z) \subset P_n(z)$.

\textbf{Thickened puzzle pieces:} For now, we have constructed bounded puzzle pieces for points in $\partial f^j(B)$, $j=0, \ldots, m-1$ and in $P(f)$ (except for points in $\Gamma_{\infty}$). In order to use techniques such as hyperbolic metric and conformal moduli, we need to make the deeper puzzle pieces compactly contained in shallower ones. This method is called `thickeness' (see \cite{Mil00}).

First note that it is possible for $P_{n+1}(f^n(z))$ to share a common boundary
with $P_{n}(f^n(z))$, only if $Q_0(z)$ and $Q_1(z)$ shares a common boundary, which consists of an internal ray $R(\theta)$, a dynamic ray $g_{\underline{s}}$ and their common landing point $\alpha$. 

The periodic point $\alpha $ is repelling by our construction. So we can choose a neighborhood of $\alpha$ given by Koening's theorem. Without loss of generality, we can choose a disk $D(\alpha, \epsilon)$. Suppose that $\alpha \in \partial f^{j_0}(B)$, $0 \le j_0 \le m-1$. Since $f^{j_0}(B)$ is an attracting or parabolic basin, the repelling periodic points are dense on $\partial f^{j_0}(B)$ (\cite{PZ94}). Thus there is another repelling periodic point $\alpha' \in D(\alpha, \epsilon) \setminus P_0(z)$. By Lemma \ref{lm: landing theorem}, there is an internal ray $R(\theta')$, and a dynamic ray $g_{\underline{s'}}$ landing at $\alpha'$. These two curve $R(\theta')$ and $g_{\underline{s'}}$ intersect $\partial D(\alpha, \epsilon)$ at two points $w,w'$. We replace the original boundary $R(\theta)$, $g_{\underline{s}}$ by $R(\theta')$, $g_{\underline{s'}}$ and an arc of $\partial D(\alpha, \epsilon)$ connecting $w, w'$, which is not in $P_0(z)$. Denote by $\tilde{P}_0(z)$ the thickened puzzle piece. Then $Q_0(z) \subset \tilde{Q}_0(z)$. Repeating the previous `cutting down' procedure, we obtain thickened bounded puzzle pieces (denoted by $\hat{P}_n$) satisfying:
\begin{itemize}
    \item $\hat{P}_{n+1}(z) \Subset \hat{P}_n(z)$, $n \ge 1$;
    \item If $\hat{P}_n(z)$ contains a critical point, then $P_n(z)$ already contains the same critical point.
\end{itemize}

For $z \in \partial B$, $z \notin \Gamma_{\infty}$, let $\mathrm{Imp}(z)=\bigcap_{n \ge 1} \overline{\hat{P}_n(z)}$. To prove the local connectivity of $\partial B$ at $z$, it suffices to show that $\mathrm{Imp}(z) \cap \partial B=\{z\}$. While for $z \in \Gamma_{\infty}$, we can always choose another set of internal rays and dynamic rays to construct another collection of puzzle pieces, such that $\hat{P}_n(z)$, $n \ge 1$ are well-defined.

\section{Local connectivity of boundaries of Fatou components}\label{sec: local connectivity of bounded Fatou components}
Let $z \in \partial B$, $z \notin \Gamma_{\infty}$. In this section, we prove that $\mathrm{Imp}(z) \cap \partial B=\{z\}$. 

First we prove that the impression of the unique critical value $-v$ in $J(f)$ intersects with $\partial f^j(B)$, $0 \le j \le m-1$, at a single point. To study the iteration of $z \in J(f)$ with $\hat{P}_n(z)$ well-defined for $n \ge 1$, define the tableau of $z$ by $T(z)=(\hat{P}_n(f^l(z)))_{n \ge 1, l \ge 0}$. For $l \ge 0$, let 
$$d_l=\max\{n \ge 1: \hat{P}_n(f^l(z)) \text{ contains a critical point}\}.$$ 
Since there is a unique critical value $-v$ in $J(f)$, $T(-v)$ has only three possibilities:
\begin{enumerate}
    \item $\sup_{l \ge 0} d_l <\infty$;
    \item for every $l \ge 0$, $d_l<\infty$, but $\sup_{l \ge 0} d_l=\infty$;
    \item there exists $l_0 \ge 0$, such that $d_{l_0}=\infty$.
\end{enumerate}

If $d_l<\infty$ for all $l \ge 0$, then using hyperbolic metric and Yoccoz's moduli method (see \cite{Hub93,Mil00,QWY12}), it is easy to show that $\mathrm{Imp}(-v) =\{-v\}$. 
\begin{lemma}\label{lm: impression is a single point-finite depth case}
    If $d_l<\infty$ for every $l \ge 0$, then $\mathrm{Imp}(-v)=\{-v\}$. 
\end{lemma}

Now we suppose that there exists $l_0 \ge 0$, such that $d_{l_0}=\infty$ for the tableau $T(-v)$. Then $T(-v)$ is periodic, i.e. there is $p \ge 1$, such that $\hat{P}_n(f^p(-v))=\hat{P}_n(-v)$ for all $n \ge 1$. Indeed, if $\hat{P}_n(f^{l_0}(-v))=\hat{P}_n(c)$ for some critical point $c$ and all $n \ge 1$, then $\hat{P}_n(f^{l_0+1}(-v))=\hat{P}_n(-v)$ for all $n \ge 1$. 
\begin{lemma}\label{lm: periodic critical tableau implies renormalization}
    For the tableau $T(-v)$, if $d_{l_0}=\infty$ for some $l_0 \ge 0$, then there is a renormalization of $f$. 
\end{lemma}

\begin{remark}
    For the definition of renormalization, one can refer to \cite{McM94}, \S 7.1.
\end{remark}

\begin{proof}
    By the discussion above, $\hat{P}_n(f^p(-v))=\hat{P}_n(-v)$ for some $p \ge 1$. Let $p$ be the minimal integer satisfying this property. Thus there exists a critical point $c$, such that $\hat{P}_n(c)=\hat{P}_n(f^{p-1}(-v)) =\hat{P}_n(f^p(c))$, $n \ge 1$. Choose $n_0 \ge 1$ sufficiently large, such that $\hat{P}_{n_0}(f^l(c))$, $1 \le l \le p-1$, contains no critical points of $f$. We claim that $f^p: \hat{P}_{n_0+p}(c) \to \hat{P}_{n_0}(c)$ is a renormalization of $f$. 

    First of all, $\hat{P}_{n_0+p}(c)$ and $\hat{P}_{n_0}(c)$ are both topological disks with $\hat{P}_{n_0+p}(c) \Subset \hat{P}_{n_0}(c)$. Secondly, $f^p: \hat{P}_{n_0+p}(c) \to \hat{P}_{n_0}(c)$ is a proper mapping since $f$ is an analytic function and $c$ is a critical point. $\mathrm{deg}(f^p:\hat{P}_{n_0+p}(c) \to \hat{P}_{n_0}(c))=2$ since $\hat{P}_{n_0+p-l}(f^l(c))$, $1 \le l \le p-1$, contains no critical points of $f$. Thus $f^p: \hat{P}_{n_0+p}(c) \to \hat{P}_{n_0}(c)$ is a quadratic-like mapping. Since $\hat{P}_n(c)=\hat{P}_{n}(f^p(c))$, $n \ge 1$, the orbit of $c$ under $f^p$ remains in $\hat{P}_{n_0+p}(c)$. Thus $f^p: \hat{P}_{n_0+p}(c) \to \hat{P}_{n_0}(c)$ has connected small filled-in Julia set. This implies that $f^p: \hat{P}_{n_0+p}(c) \to \hat{P}_{n_0}(c)$ is a renormalization of $f$.
\end{proof}

If $f$ is renormalizable, then the impression $\mathrm{Imp}(-v)$ is no longer a single point, but rather the small filled-in Julia set $K_c$ with respect to the renormalization $f^p: \hat{P}_{n_0+p}(c) \to \hat{P}_{n_0}(c)$. In this case, we can show that $K_c$ intersects every $\partial f^j(B)$, $0 \le j \le m-1$, at most one point. 

We may assume that $K_c \cap \partial B \neq \varnothing$. 
\begin{lemma}\label{lm: intersection of the small filled-in Julia set and the boundary of B is a single point}
    The small filled-in Julia set $K_c$ intersects $\partial B$ at a single point $\beta_c$.  
\end{lemma}

The following lemma can be found in \cite{McM94}:
\begin{lemma}\label{lm: invariant ray lands at beta fixed point}
    Suppose that $p: \mathbb{C} \to \mathbb{C}$ is a quadratic polynomial with $K(p)$ connected. Suppose that $\gamma:(0,+\infty) \to \mathbb{C} \setminus K(p)$ is a curve, satisfying $\gamma \subset p(\gamma)$, and the limit $\lim_{t \to 0}\gamma(t)=z_0$ exists. Then $z_0$ is exactly the $\beta$-fixed point of $p$. 
\end{lemma}

Before proving Lemma \ref{lm: intersection of the small filled-in Julia set and the boundary of B is a single point}, in order to describe the location of dynamic rays, we introduce an order in the symbil space $\Sigma$. For $(j_1, k_1)$, $(j_2, k_2) \in \{0,1\} \times \mathbb{Z}$, define an order `<' by $(j_1, k_1) < (j_2, k_2)$ if and only if one of the following holds:
\begin{enumerate}
    \item $j_1=1$, $j_2=0$, or
    \item $j_1=j_2$, and $(-1)^{j_1} k_1 < (-1)^{j_2} k_2$.
\end{enumerate} 
Then the order $`<'$ induces the lexicograpgic order $`\prec'$ on $\Sigma$. Equivalently, the two sequences $\underline{s}_1 \prec \underline{s}_2$, if and only if one of the following holds:
\begin{enumerate}
    \item $g_{\underline{s}_1}$ is eventually contained in the left half plane $\{z \in \mathbb{C}: \mathrm{Re}(z-u)<0\}$, while $g_{\underline{s}_2}$ is eventually contained in the right half plane $\{z \in \mathbb{C}: \mathrm{Re}(z-u)>0\}$, or
    \item $g_{\underline{s}_1}$ and $g_{\underline{s}_2}$ are both eventually contained in the left half plane $\{z \in \mathbb{C}: \mathrm{Re}(z-u)<0\}$, and $g_{\underline{s}_1}$ is above $g_{\underline{s}_2}$, or 
    \item $g_{\underline{s}_1}$ and $g_{\underline{s}_2}$ are both eventually contained in the right half plane $\{z \in \mathbb{C}: \mathrm{Re}(z-u)>0\}$, and $g_{\underline{s}_2}$ is above $g_{\underline{s}_1}$.
\end{enumerate} 
Here `above' is in the following sense: if $g_{\underline{s}}$ is eventually contained in the right half plane, for sufficiently large $M>0$, the unbounded component of $g_{\underline{s}} \cap \{z \in \mathbb{C}: \mathrm{Re} \; z>M\}$ separates $\{z \in \mathbb{C}: \mathrm{Re} \; z>M\}$ into the upper part $U_+$ and the lower part $U_-$. $g_{\underline{s'}}$ is said to be above $g_{\underline{s}}$ if the unbounded component of $g_{\underline{s'}} \cap \{z \in \mathbb{C}: \mathrm{Re} \; z>M\}$ is contained in $U_+$. When $g_{\underline{s}}$ is eventually contained in the left half plane, we can define `above' analogously. 

Now we prove Lemma \ref{lm: intersection of the small filled-in Julia set and the boundary of B is a single point}.

\begin{proof}[Proof of Lemma \ref{lm: intersection of the small filled-in Julia set and the boundary of B is a single point}]
    Recall that $f^p: \hat{P}_{n_0+p}(c) \to \hat{P}_{n_0}(c)$ is a proper mapping of degree $2$. Let $R_B(\theta^{\pm}_k)$ be the internal rays contained in $\partial \hat{P}_{n_0+kp} \cap B$ with $\theta^{-}_k<\theta^+_k$, and $g_{\underline{s}^{\pm}_k}$ be the dynamic rays in $\partial \hat{P}_{n_0+kp}(c)$, landing at the same point as $R_B(\theta^{\pm}_k)$. Suppose that $p=rm$, $r \ge 1$. Then $|\theta^+_k-\theta^-_k|=2^r|\theta^+_{k+1}-\theta^-_{k+1}|$. Therefore, $\lim_{k \to \infty} \theta^+_k = \lim_{k \to \infty} \theta^-_k = \theta_0$, and $R_B(\theta_0)=f^p(R_B(\theta_0))$. By Lemma \ref{lm: invariant ray lands at beta fixed point}, $R_B(\theta_0)$ lands at $\beta_c$. 
    
    By Lemma \ref{lm: monotoncity of external addresses}, when $k$ is sufficiently large, $\underline{s}^-_k \prec \underline{s}^-_{k+1} \prec \underline{s}^+_{k+1} \prec \underline{s}^+_k$, and there exists $\underline{s}^{\pm}$ such that $\lim_{k \to \infty} \underline{s}^{\pm}_k = \underline{s}^{\pm}$. By Lemma \ref{lm: landing theorem}, $g_{\underline{s}^{\pm}}$ land at periodic points $y_{\pm} \in K_c$. By \cite{LQ07}, Lemma 4.7, $y_{\pm}=\beta_c$. Then by \cite{RY08}, Proposition 2, $\overline{g_{\underline{s}^{+}} \cup g_{\underline{s}^{-}}}$ separates $\partial B$ and $K_c \setminus \{\beta_c\}$, i.e. $K_c \cap \partial B =\{\beta_c\}$. 
\end{proof}

\begin{lemma}\label{lm: monotoncity of external addresses}
    Let $g_{\underline{s}_{k}^{\pm}} \subset \partial \tilde{P}_{n_0+kp}(c)$ be defined as in the proof of Lemma \ref{lm: intersection of the small filled-in Julia set and the boundary of B is a single point}. There exists $K \ge 0$, such that for $k \ge K$, $\{\underline{s}^+_k\}$, $\{\underline{s}^-_k\}$are both monotone with respect to $k$. And there exists $\underline{s}^{\pm}$, such that $\lim_{k \to \infty} \underline{s}^{\pm}_k=\underline{s}^{\pm}$, and $\sigma^p(\underline{s}^{\pm})=\underline{s}^{\pm}$. Here the limit is taken in the sense of the product topology on $\Sigma$. 
\end{lemma}
\begin{proof}
    As in the proof of Lemma \ref{lm: intersection of the small filled-in Julia set and the boundary of B is a single point}, let $R_B(\theta^{\pm}_k)$ be the internal rays, which land at the same point as $g_{\underline{s}^{\pm}_k}$. 

    First prove $\underline{s}^{\pm}_k$ are monotone with respect to $k$. Consider $\underline{s}^+_k$, $\underline{s}^-_k$ is similar. 

    \textbf{Case 1.} There exists $k_0$ such that $g_{\underline{s}^+_{k_0}}$ and $g_{\underline{s}^-_{k_0}}$ are in the same left or right half plane. Without loss of generality, we may assume that they are both eventually contained in the right half plane $\{z \in \mathbb{C}: \mathrm{Re}(z-u)>0\}$. In this case, $\overline{R_B(\theta^+_{k_0}) \cup g_{\underline{s}^+_{k_0}}}$ divides the half plane $\{z \in \mathbb{C}: \mathrm{Re}(z-u)>0\}$ into two components: the upper one $U^+_+$ and the lower one $U^+_-$. Similarly, $\overline{R_B(\theta^+_{k_0}) \cup g_{\underline{s}^+_{k_0}}}$ divides $\{z \in \mathbb{C}: \mathrm{Re}(z-u)>0\}$ into two components: the upper one $U^-_+$, and the lower one $U^-_-$. $R_B(\theta^{\pm}_{k_0+1}) \subset U^+_- \cap U^-_+$ since $\theta^-_{k_0} < \theta^{\pm}_{k_0+1}<\theta^+_{k_0}$. The two curves $\overline{R_B(\theta^+_{k_0}) \cup g_{\underline{s}^+_{k_0}}}$ and $\overline{R_B(\theta^+_{k_0+1}) \cup g_{\underline{s}^+_{k_0+1}}}$ cannot have intersection except for the possible superattracting point in $B$. Thus $g_{\underline{s}^+_{k_0+1}} \subset U^+_-$. In other words, $\underline{s}^+_{k_0+1} \prec \underline{s}^+_{k_0}$.
    
    Inductively, $g_{\underline{s}^{\pm}_k}$ are always in the same half plane $\{z \in \mathbb{C}: \mathrm{Re}(z-u)>0\}$, and $\underline{s}^+_{k+1} \prec \underline{s}^+_k$ for all $k \ge k_0$. 

    \textbf{Case 2.} For all $k \ge 0$, $g_{\underline{s}^+_k}$ and $g_{\underline{s}^-_k}$ are in different left and right half planes. We may assume that $g_{\underline{s}^+_k}$ (resp. $g_{\underline{s}^-_k}$) is eventually contained in $\{z \in \mathbb{C}: \mathrm{Re}(z-u)<0\}$ (resp. $\{z \in \mathbb{C}: \mathrm{Re} (z-u)>0\}$). Let $U_+$ be the upper component of $\{z \in \mathbb{C}: \mathrm{Re}(z-u)<0\} \setminus \overline{R_B(\theta^+_k) \cup g_{\underline{s}^+_k}}$. By the same reason in Case 1, $g_{\underline{s}^+_{k+1}} \subset U_+$, i.e. $\underline{s}^+_{k+1} \prec \underline{s}^+_k$. 
    
    Then we prove $\{\underline{s}^{\pm}_k\}$ converges with respect to the product topology on $\Sigma$. 

    It is easy to see that the product topology on $\Sigma$ is equivalent to the topology given by the metric:
    \begin{equation*}
        d(\underline{s}, \underline{t})=\frac{1}{2^k}, 
    \end{equation*}
    where $\underline{s}=(s_0, s_1, \ldots)$, $\underline{t}=(t_0, t_1, \ldots)$, and $k=\min\{n \ge 0: s_n \neq t_n\}$. 

    Recall that $\underline{s}^+_k=\sigma^p(\underline{s}^+_{k+1})$. This implies that $\underline{s}^+_{k+1}$ is obtained by adding $p$ entries before the first entry of $\underline{s}^+_k$. It suffices to show that the $p$ entries are independent of $k$, so that 
    \begin{equation*}
        d(\underline{s}^+_{k+2}, \underline{s}^+_{k+1}) \le \frac{1}{2^p} d(\underline{s}^+_{k+1}, \underline{s}^+_k). 
    \end{equation*}

    We may assume $p=m=1$. 
    
    \textbf{Case 1.} There exists $k_0$ such that $g_{\underline{s}^{\pm}_{k_0}}$ are in the same left or right half plane. By the discussion above, $g_{\underline{s}^{\pm}_{k}}$ are eventually contained in the same left or right half plane for all $k \ge k_0$. Again we assume that they are both eventually contained in the right half plane. Suppose $\underline{s}^+_{k_0}=(s_0, s_1, \ldots)$, with $s_0=(0, l_{k_0})$. Then $\underline{s}^{+}_{k_0+1}=(s'_0, s_0, s_1, \ldots)$, with $s'_0=(0, l_{k_0+1})$. Let $\underline{s}'=\sigma(\underline{s}^+_k)$. Then $g_{\underline{s}^+_{k_0}+1}$ is contained in the strip bounded by $\overline{R_B(\theta^+_k) \cup g_{\underline{s}^+_{k_0}}}$, $\overline{R_B(\theta^+_k) \cup g_{\underline{s}^+_{k_0}}}-2\pi i$, and $\{z \in \mathbb{C}: \mathrm{Re}(z-u)=0\}$. Thus $l_{k_0+1}$ equals to either $l_{k_0}$ or $l_{k_0}-1$. 

    Inductively, the first entry $s^{k+1}_{0}=(0,l_{k+1})$ of $\underline{s}^+_{k+1}$ is either equal to $s^k_0$, or $l_{k+1}=l_k-1$ for all $k \ge k_0$. However, since they all land at $\partial B$, $g_{\underline{s}^+_{k}}$, $k \ge k_0+1$ are contained in the half strip  bounded by $\overline{R_B(\theta^+_k) \cup g_{\underline{s}^+_{k_0}}}$, $\overline{R_B(\theta^+_k) \cup g_{\underline{s}^+_{k_0}}}-2\pi i$, and $\{z \in \mathbb{C}: \mathrm{Re}(z-u)=0\}$. Thus there exists $k'_0 \ge k_0$ such that $\underline{s}^+_k$, $k \ge k'_0$ have the same first entry. 

    \textbf{Case 2.} For all $k \ge 0$, $g_{\underline{s}^+_k}$ and $g_{\underline{s}^-_k}$ are eventually contained in different left and right half planes. We may assume that $g_{\underline{s}^+_k} \subset \{z \in \mathbb{C}: \mathrm{Re}(z-u)<0\}$ and $g_{\underline{s}^-_k} \subset \{z \in \mathbb{C}: \mathrm{Re} (z-u)>0\}$. In this case, $\underline{s}^+_{k}$, $k \ge k_0+1$ are contained in the half strip bounded by $\overline{R_B(\theta^+_k) \cup g_{\underline{s}^+_{k_0}}}$, $\overline{R_B(\theta^+_k) \cup g_{\underline{s}^+_{k_0}}}+2\pi i$, and $\{z \in \mathbb{C}: \mathrm{Re}(z-u)=0\}$. Similar to Case 1, there exists $k'_0 \ge 0$, such that the first entries of $\underline{s}^+_k$, $k \ge k'_0$ are the same. 
\end{proof}

At the end of Section \ref{sec: local connectivity of bounded Fatou components}, we prove Theorem \ref{thm: bounded Fatou components are Jordan domains}. 

\begin{proposition}\label{prop: bounded periodic Fatou components are Jordan domains}
    Suppose that $P(f)$ is bounded, $U$ is a bounded periodic Fatou component of $f$, which is not a Siegel disk. Then $U$ is a Jordan domain. 
\end{proposition}
\begin{proof}
    Let $B=f^q(U)$ be the Fatou component in the cycle of $U$, which contains a critical point. Let $z \in \partial B \setminus \Gamma_{\infty}$ so that $\hat{P}_n(z)$, $n \ge 1$, are well-defined. Consider $T(z)$. Let 
    \begin{equation*}
        d_l(z)=\min\{n \ge 1: \hat{P}_n(f^l(z)) \text{ contains a critical point}\}.
    \end{equation*}
    
    \textbf{Case 1. } $d_l<\infty$, $l \ge 0$ for the tableau $T(-v)$. In this case, no matter $d_l(z)$ is finite or infinite, $\mathrm{Imp}(z)=\{z\}$ by the argument of hyperbolic metric and conformal moduli. 

    \textbf{Case 2. } For the tableau $T(-v)$, $d_l=\infty$ for some $l \ge 0$. In this case, by Lemma \ref{lm: periodic critical tableau implies renormalization}, $f$ is renormalizable. By Lemma \ref{lm: intersection of the small filled-in Julia set and the boundary of B is a single point}, the small filled-in Julia set $K_c$ intersects $\partial B$ at a single point.

    For $z \in \partial B \setminus \Gamma_{\infty}$, if $d_l(z)<\infty$ for all $l \ge 0$, then $\mathrm{Imp}(z)=\{z\}$ again by the argument of hyperbolic metric and conformal moduli. If $d_l(z)=\infty$ for some $l \ge 0$, then $f^l(z)=\beta_c$, where $\beta_c$ is defined in Lemma \ref{lm: intersection of the small filled-in Julia set and the boundary of B is a single point}, and $\mathrm{Imp}(f^l(z))=\{f^l(z)\}$. $f^l(\mathrm{Imp}(z))=\mathrm{Imp}(f^l(z))$. Hence $\mathrm{Imp}(z)$ is a single point.

    If $z \in \partial B \cap \Gamma_{\infty}$, we can always switch to another collection of internal rays and dynamic rays to construct a completely invariant graph $\Gamma'_{\infty}$, so that $z \notin \Gamma'_{\infty}$. 

    The procedure above implies that $\partial B$ is locally connected. By the Maximum Modulus Principle, $B$ is simply connected, i.e. $B$ is a Jordan domain. Since $f^q: \partial U \to \partial B$ is a local homeomorphism, $\partial U$ is locally connected, and $U$ is a Jordan domain. 
\end{proof}

\begin{proof}[Proof of Theorem \ref{thm: bounded Fatou components are Jordan domains}]
    Let $U$ be a bounded Fatou component of $f$, which is not eventually mapped to a Siegel disk. By \cite{EL92}, Theorem 1 \& Theorem 3, there exists minimal $n \ge 0$ such that $f^n(U)$ is a periodic Fatou component of $f$, and $f^n(U)$ is bounded. By Proposition \ref{prop: bounded periodic Fatou components are Jordan domains}, $f^n(U)$ is a Jordan domain. Since $f^n: \partial U \to \partial f^n(U)$ is a local homeomorphism, $\partial U$ is also locally connected, and $U$ is a Jordan domain by the Maximum Modulus Principle.     
\end{proof}

\section{Renormalization of cosine functions}\label{sec: renormalization of cosine functions}
In this section, we consider cosine functions with an escaping critical value. As before, $f$ is a cosine function, with critical values $\pm v$, critical points $u_k=u+k\pi i$, $k \in \mathbb{Z}$, such that $\mathrm{Im} \; u \in (-\pi ,\pi]$, $f(u_k)=(-1)^k v$. Suppose that $v \in F(f)$, and $-v \in I(f)$. 

\begin{proof}[Proof of Theorem \ref{thm: cosine functions with an escaping critical value is renormalizable}]
    By the assumption in Theorem \ref{thm: cosine functions with an escaping critical value is renormalizable}, the possibilities of periodic components of $F(f)$ are the basin of attracting or parabolic periodic points or Siegel disks, and the critical value $v$ lies in the closure of a periodic Fatou component. 

We may assume that $v$ belongs to the closure of a periodic Fatou component and $-v$ escapes to $\infty$. Since $-v \in I(f)$, there is a dynamic ray $g_{\underline{s}}$ containing $-v$. Thus for every critical point $u_{2k+1}$, there are exactly two dynamic rays $g_{\underline{s}^{\pm}_k}$, either landing or crash on $u_{2k+1}$. The rays $g_{\underline{s}^{\pm}_k}$ are symmetric about $u_{2k+1}$, and $g_{\underline{s}^{\pm}_{k+1}}=g_{\underline{s}^{\pm}_k}+2\pi i$. Let $S_k$ be the strip bounded by $\overline{g_{\underline{s}^+_k} \cup g_{\underline{s}^-_k}}$ and $\overline{g_{\underline{s}^+_{k+1}} \cup g_{\underline{s}^-_{k+1}}}$. Then $S_k$, $k \in \mathbb{Z}$ gives a partition of $\mathbb{C}$, and $S_k$ contains a unique critical point $u_{2k+2}$.

Recall that there are half strips $P_{j,k}$, with $[u_{2k}, u_{2k+2}] \subset \partial P_{j,k}$, which are mapped onto $\mathbb{C} \setminus \Gamma$. By Lemma \ref{lm: properties of dynamic rays}, $S_{k_0}$ is contained in finitely many half strips $P_{j,k}$. Thus it is contained in a horizontal strip $\{z \in \mathbb{C}: M_1 \le \mathrm{Im} \; z \le M_2\}$, where $M_1<M_2$ are constants and $M_2-M_1 \ge 2\pi$.

Let 
$$R_M=\{z \in S_{k_0}: |\mathrm{Re}(z-u_{2k_0+2})| < M\},$$
where $M>0$. Then $\{z \in \mathbb{C}: |\mathrm{Re} \; (z-u_0)| <M\}$ is mapped to an ellipse $E_M$ whose foci are $\pm v$. The lengths of the major axis and the minor axis of $E_M$ are $|v|(e^M+e^{-M})$ and $|v|(e^M-e^{-M})$ respectively. So for $M$ sufficiently large, $R_M$ is compactly contained in $E_M$. 

The upper and lower bound of $S_{k_0}$ are mapped to  $g_{\underline{s}}$ connecting the ciritcal value $-v$ and $\infty$. If $-v \notin R_M$, then $f:R_M \to E_M \setminus g_{\underline{s}}$ is a proper mapping of degree $2$, which is a renormalization of $f$. 

However, if $-v \in R_M$, in order to construct a proper mapping, we need to remove $g_{\underline{s}}$ and its forward images under the iteration of $f$. Since $-v$ escapes to $\infty$, it leaves $R_M$ eventually. Therefore, there exists $N \ge 1$, such that $f^{N-1}(-v) \in R_M$, and $f^N(-v) \notin R_M$.

In this case, let $U_M= R_M \setminus \left( \bigcup_{j=0}^{N-1} f^j(g_{\underline{s}})\right)$, $V_M=f(U_M)$. Then $f: U_M \to V_M$ is a proper mapping of degree $2$, with $U_M \Subset V_M$. This implies $f:U_M \to V_M$ is a renormalization. And this completes the proof of Theorem \ref{thm: cosine functions with an escaping critical value is renormalizable}. 
\end{proof}

\begin{proof}[Proof of Corollary \ref{cor: Fatou components are locally connected in renormalizable case}]
    $f$ has two critical values $\pm v$. Without loss of generality, we may assume that $-v$ escapes to $\infty$, and $v$ belongs to the closure of a periodic Fatou component. As in the proof of Theorem \ref{thm: bounded Fatou components are Jordan domains}, it will be sufficient to show that all periodic Fatou components are Jordan domains. 

    By Lemma \ref{lm: dichotomy of the positions of critical points}, all Fatou components are bounded. By Theorem \ref{thm: cosine functions with an escaping critical value is renormalizable}, $f$ is renormalizable. By the Straightening Theorem in \cite{DH85}, the periodic Fatou components are homeomorphic to the Fatou components of a quadratic polynomial, which has no Siegel disks. By \cite{RY08}, Theorem 1, these Fatou components are Jordan domains. Thus the periodic Fatou components of $f$ are Jordan domains.
\end{proof}

\section{Local connectivity of Julia sets}\label{sec: local connectivity of Julia sets}
In this section, we prove Theorem \ref{thm: local connectivity of Julia set}. Recall the charaterization of local connectivity:

\begin{lemma}[\cite{Why42}, Chapter VI, Theorem 4.4]\label{lm: charaterization of local connectivity}
    A connected compact $E \subset \hat{\mathbb{C}}$ is locally connected, if and only if 
    \begin{enumerate}
        \item the boundary of every component of $\hat{\mathbb{C}} \setminus E$ is locally connected, and 
        \item for every $\epsilon>0$, there are only finitely many components of $\hat{C} \setminus E$, whose spherical diameter is greater than $\epsilon$. 
    \end{enumerate}
\end{lemma}

Since all Fatou components of $f$ are preperiodic, for a Fatou component $U$, we can define an non-negative integer $m(U)=\min\{n \ge 0: f^n(U) \text{ is periodic}\}$ to be the time such that the iteration of $U$ becomes periodic. The following lemma implies that if this time is bounded for a sequence of Fatou components, then the spherical diameters of these components tend to $0$. 

\begin{lemma}\label{lm: mapped to periodic components in finitely many steps}
    Let $\{U_n\}_{n \ge 1}$ be a sequence of Fatou components. Suppse that $m(U_n)\le m$ for all $n \ge 1$. Then $\mathrm{diam}_s \; U_n \to 0$ as $n \to \infty$. 
\end{lemma}
\begin{proof}
    Since $f$ has only finitely many periodic Fatou components, and the time $m(U_n)$ are bounded, we can divide $\{U_n\}$ into finitely many subsequences $\{U^j_n\}_{0 \le j \le m, n \ge 1}$, such that $m(U^j_n)=j$, and $f^j(U^j_n)$, $n \ge 1$ are all the same. To simplify the notation, we may assume that $m(U_n)=m$, and $f^m(U_n)$, $n \ge 1$ are all the same, denoted by $U$. Let $g=f^m$, which is an entire function.  Then $U_n$ are the components of $g^{-1}(U)$. Thus $\{U_n\}$ can only accumulate at $\infty$. For $\epsilon >0$, let $D_{\epsilon}$ be the spherical disk of center $\infty$ and radius $\epsilon/2$. Then there are only finitely many $U_n$ which intersects with $\hat{C} \setminus D_{\epsilon}$. Otherwise, $\{U_n\}$ will accumulate at some point other than $\infty$. Therefore, $\mathrm{diam}_s \; U_n \to 0$ as $n \to \infty$. 
\end{proof}

\begin{proof}[Proof of Theorem \ref{thm: local connectivity of Julia set}]
    Let $f$ be as in Theorem \ref{thm: local connectivity of Julia set}. 

    \textbf{Case 1.} $-v \in F(f)$ and is in a different Fatou component from $v$. By the classification of periodic Fatou components, $-v$ is attracted by either an attracting or a parabolic periodic point. Thus $f$ is either hyperbolic or geometrically finite so that $J(f) \cup \{\infty\}$ is locally connected by the results in \cite{ARS22, BFR15}.
    
    \textbf{Case 2.} $-v$ escapes to $\infty$. By Corollary \ref{cor: Fatou components are locally connected in renormalizable case}, all Fatou components are Jordan domains. It remains to show that the spherical diameters of Fatou components tend to $0$. Let $V$ be a Fatou component, which is strictly preperiodic and such that $f(V)$ is periodic. We claim that $f$ is semihyperbolic on $\partial V$, where the definition of semihyperbolicity can be found in \cite{BM02}. By Lemma \ref{lm: mapped to periodic components in finitely many steps}, the spherical diameters of such $V$s tend to $0$, in particular, the spherical diameters are bounded. Thus the spherical diameters of Fatou components tend to $0$ by the same argument as in \cite{BM02}.
    
    Indeed, if $f$ has no parabolic periodic points, then $\partial V \cap P(f) =\varnothing$ since the orbit of $v$ tends to an attracting cycle and the orbit of the other critical value $-v$ escapes to $\infty$. If $f$ does have parabolic periodic points, it will be sufficient to show that $\partial V$ contains no parabolic periodic points. And again $\partial V \cap P(f) =\varnothing$. 

    Suppose by contradiction that $\partial V$ contains a parabolic periodic point $z_0$, whose period is $p \ge 1$. Then $z_0$ must be on the boundary of some perioic Fatou component $U$, whose period is $kp$, $k \ge 1$. This implies that $z_0$ is the common boundary point of exact $k$ periodic Fatou components.
    
    If $k=1$, then $U=f^p(U)=f^p(V)$, which implies that $z_0$ is a critical point of $f$. This is impossible. If $k \ge 2$, then $\tilde{U}=f^p(V)$ is a periodic Fatou component whose boundary contains $z_0$, and $f^{(k-1)p}(\tilde{U})$ is also a preimage of $\tilde{U}$. Let $N_{z_0}$ be a sufficiently small neighborhood of $z_0$, on which $f^p$ is conformal. But $V \cap N_{z_0}$ and $f^{(k-1)p}(\tilde{U}) \cap N_{z_0}$ are both mapped onto $\tilde{U} \cap f^p(N_{z_0})$, cntradicting that $f^p$ is conformal on $N_{z_0}$. Therefore, $\partial V$ contains no parabolic periodic points.    

    \textbf{Case 3.} $-v$ has bounded orbit and there exists an integer $N \ge 0$, such that for every periodic component $U$, every strictly preperiodic component $V$ of $f^{-N}(U)$, $\overline{V}$ is disjoint from the $\omega$-limit set of $-v$. Since the orbit of $-v$ is bounded, all Fatou components are Jordan domains by Theorem \ref{thm: bounded Fatou components are Jordan domains}. It remains to show that the spherical diameters of Fatou components tend to $0$. 

    As in Case 2, $\partial V$ contains no parabolic periodic points. Thus $\partial V \cap P(f)=\varnothing$, and $f$ is semihyperbolic on $\partial V$. By Lemma \ref{lm: mapped to periodic components in finitely many steps}, the spherical diameters of components of $f^{-N}(U)$ are bounded. Thus the spherical diameters of Fatou components tend to $0$. 
\end{proof}

\bibliographystyle{plain}
\bibliography{Boundaries-of-Fatou-components-of-cosine-functions.bib}

\end{document}